\documentclass[12pt,a4paper,reqno]{amsart}
\usepackage[all]{xy}
\usepackage{amsmath}
\usepackage{amsfonts}
\usepackage{amssymb,amsthm,amsfonts,amsthm,latexsym,enumerate,url,cases}
\numberwithin{equation}{section}
\usepackage{mathrsfs}
\usepackage{hyperref}
\usepackage{extarrows}
\hypersetup{colorlinks=true,citecolor=blue,linkcolor=blue,urlcolor=blue}
\def\pmod #1{\ ({\rm{mod}}\ #1)}

\newtheorem{theorem*}{Theorem}
\newtheorem{lemma*}{Lemma}
\theoremstyle{plain}

\newtheorem{theorem}{Theorem}
\newtheorem{lemma}{Lemma}

\newtheorem{proposition}{Proposition}

\newtheorem{problem}{Problem}
\theoremstyle{definition}

\newtheorem{remark}{Remark}


\begin{document}

\title
[{On three problems of Y.--G. Chen}] {On three problems of Y.--G. Chen}

\author
[Y. Ding] {Yuchen Ding}

\address{(Yuchen Ding) School of Mathematical Science,  Yangzhou University, Yangzhou 225002, People's Republic of China}
\email{ycding@yzu.edu.cn}

\thanks{*Corresponding author}

\begin{abstract}  
In this short note, we answer two questions of Chen and Ruzsa negatively and answer a problem of Ma and Chen affirmatively. 
\end{abstract}
\maketitle

\baselineskip 18pt
\section{Two problems on $r$--full numbers}
The irrationalities of certain type numbers are of great interest to mathematicians. People have known for a long time that $e$ and $\pi$ are irrational (see e.g. \cite{Murty2}). Following the formulae of Euler involving Bernoulli's numbers (see e.g. \cite{Murty2}), we know that $$\zeta(2s)=\sum_{n=1}^{\infty}\frac1{n^s}$$ are irrational numbers for all positive even integers $s$. It wasn't until 1978 that Ap\'ery \cite{Ap} gave a miraculous proof for the irrationality of $\zeta(3)$.  For an alternative proof of the irrationality of $\zeta(3)$, one can also refer to the note of Beukers \cite{Beukers}. In a remarkable paper, Rivoal \cite{Ri} showed that there are infinitely many irrational values of the Riemann zeta function $\zeta(s)$ at odd integers $s$. Rivoal \cite{Ri2} further proved that at least one of the nine numbers $\zeta(5),\zeta(7),...,\zeta(21)$ is irrational, which was then sharpen by Zudilin \cite{Zu} by displaying that at least one of the four numbers $\zeta(3),\zeta(5),\zeta(7),\zeta(9)$ is irrational. 

In another direction,  Erd\H os \cite{Erdos} proved that if $a_n\in \mathbb{N}$ and $$\lim_{n\rightarrow\infty}(a_{n+1}-a_n)=\infty,$$ 
then 
$$\sum_{n=1}^{\infty}\frac{a_n}{2^{a_n}}$$
is irrational. It is surely that the gaps between the square--free numbers do not satisfy the above requirement. Thus, Erd\H os \cite{Erdos} further conjectured that 
$$\sum_{n~\text{square--free}}\frac{n}{2^n}$$
is also irrational. This conjecture was later resolved by Chen and Ruzsa \cite{CR}.  Now, let's fix some terminologies which shall be used in the context below. If $p^r\nmid n$ for any prime $p$, then the integer $n$ is said to be $r$--free. An integer $n$ is called $r$--full if $p|n$ implies $p^r|n$ for any prime $p$. Conventionally, $2$--free and $2$--full numbers are called square--free and square--full numbers, respectively. 
In fact, Chen and Ruzsa \cite[Theorem 4]{CR} provided the following more general result.

\begin{proposition}[Chen--Ruzsa]
Let $r,\ell$ be two integers with $r,\ell\ge 2$ and $\{a_n\}$ a sequence of distinct positive integers such that each $a_n$ is either $r$--free or square--full. Then $$\sum_{n=1}^{\infty}a_n\ell^{-a_n}$$ is irrational.
\end{proposition}

Thinking about their proof, Chen and Ruzsa naturally raised the following two problems \cite[Section 3. Problems]{CR}.
\begin{problem}\label{po2}
Let $r,\ell\ge 2$ be two given integers. Then for any $k\ge 1$, is there a positive integer $m$ such that each of $\ell^m+i~(1<i<k)$ is $r$--full?
\end{problem}
\begin{problem}\label{po3}
Let $r,\ell\ge 2$ be two given integers. Then for any $k\ge 1$, is there a positive integer $m$ such that $\ell^m+k$ is $r$--full?
\end{problem}
Without doubt, Problem \ref{po2} implies Problem \ref{po3}. We shall show that the answer to Problem \ref{po3} is negative, and hence the answer to Problem \ref{po2} is negative either. That is,
\begin{theorem}\label{thm1}
Let $r,\ell\ge 2$ be two given integers. Then there exists a positive integer $k$ such that $\ell^m+k$ is not $r$--full for any positive integer $m$.
\end{theorem}
To prove Theorem \ref{thm1}, we need the following well-known result due to Dirichlet (see for example \cite[Page 34, Chapter 4]{Da}).
\begin{lemma}[Dirichlet]\label{lemma}  Let $a$ and $b$ be integers such that $(a,b)=1$, then there are infinitely many primes $p\equiv a\pmod{b}$.
\end{lemma}
\begin{proof}[Proof of Theorem \ref{thm1}]
The proof of our theorem will be divided into three cases.

{\bf Case I.} Assuming first $\ell=2$, we consider the number $2^m+10$. If $m=1$, then $2^1+10=12$ is not $2$--full, hence not $r$-full. If $m\ge 2$, then $2|2^m+10$ but $2^2\nmid 2^m+10$, which means that $2^m+10$ is not $2$--full for $m\ge 2$, hence not $r$--full either.

{\bf Case II.} We next assume that there exists a prime $p$ such that $p^2|\ell$, in which case we choose $k=p$. It is sure that $p|\ell^m+p$ but $p^2\nmid \ell^m+p$. Hence, $\ell^m+p$ is not $2$--full for any positive integers $m$, hence not $r$--full either.

{\bf Case III.} Finally, we assume that $\ell$ is a square--free number with at least one odd prime factor $q$. By Dirichlet's theorem in arithmetic progressions, i.e., Lemma \ref{lemma}, there is some integer $s\ge 2$ such that $\ell s-1$ is another prime $q_*$, from which we know that $(q_*,\ell)=1$. Let $k=\ell(q_*-1)$. Then
\begin{align}\label{eq1}
\ell^m+k=\ell(\ell^{m-1}+q_*-1).
\end{align}
If $m=1$, then $\ell+k=\ell q_*$. Thus we have $q_*|\ell+k$ but $q_*^2\nmid \ell+k$ since $(q_*,\ell)=1$.  If $m\ge 2$, then $q|\ell^m+k$ by Eq. (\ref{eq1}). Recall that $\ell$ is square--free, it follows that $q^2\nmid \ell^m+k$ since otherwise we will have $$q|\ell^{m-1}+q_*-1$$
again by Eq. (\ref{eq1}). Thus, we have $$q|q_*-1\Rightarrow q|\ell s-2,$$ from which we deduce that $q|2$. This is a contradiction with $q\ge 3$. 
Therefore, $\ell+k$ is not $2$--full for any $m$, hence not $r$--full either.
\end{proof}

\section{A problem on complete sequences}
Let $\mathbb{N}$ be the set of nonnegative integers. Suppose that $A$ is a sequence of nonnegative integers whose terms are not necessarily different. Let $P(A)$ be the set of nonnegative integers which can be represented as the sum of finitely distinct terms of $A$. The sequence $A$ is said to be complete if all sufficiently large integers belong to $P(A)$. Conventionally, we declare that $0\in P(A)$. Burr \cite{Bu} asked that which subsets $B$ of $\mathbb{N}$ are equal to $P(A)$ for a given sequence $A$. This problem remains widely open but some partial results had already been obtained, see for example \cite{Bi,Br,CF1,FL,Gr,He2,He3}. For $S=\{s_1,s_2,s_3,...\}$ and $\alpha>0$, let $$S_\alpha=\left\{\lfloor \alpha s_1\rfloor,\lfloor \alpha s_2\rfloor,...\right\}.$$ Typically, Hegyv\'ari \cite{He1} proved the following interesting result. 

\begin{proposition}[Hegyv\'ari] \label{proposition2-1}
Let $S=\{s_1<s_2<\cdot\cdot\cdot\}$ be a sequence of positive integers such that $$\lim_{n\rightarrow\infty}(s_{n+1}-s_n)=\infty$$ and $s_{n+1}<\gamma s_n$ for all sufficiently large integers $n$, where $1<\gamma<2$.  Suppose that $S_\alpha$ is complete for some $\alpha>0$, then there exists a positive number $\delta$ such that $S_\beta$ are complete for all $\beta\in[\alpha,\alpha+\delta).$
\end{proposition}

The unnecessary condition $$\lim_{n\rightarrow\infty}(s_{n+1}-s_n)=\infty$$ was then removed by Chen and Fang \cite{CF2}. In a subsequent article, Ma and Chen \cite{MC} proved that the prerequisite is always valid in Hegyv\'ari's Theorem. Motivated by

\begin{proposition}[Ma--Chen] \label{proposition2-2}
Let $S=\{s_1<s_2<\cdot\cdot\cdot\}$ be a sequence of positive integers such that $s_{n+1}\le 2s_n$ for all sufficiently large integers $n$. Suppose that $P(S_\alpha)\neq \mathbb{N}$ for any $0<\alpha<1$. Then for any positive integer $w$, there exist a real number $0<\alpha<1$ and a positive integer $n$ such that
$$\left\lfloor \alpha s_{n+i}\right\rfloor=2^i \quad (i=0,1,...,w).$$
\end{proposition}

Ma and Chen posed the following problem \cite[Problem 3.2]{MC}.
\begin{problem}\label{problem1} Does there exist a sequence $S=\{s_1,s_2,...\}$ of positive integers with $s_n<s_{n+1}\le 2s_n$ and a positive integer $m$ such that for each $0<\alpha<1$, 
$$\{2^i:0\le i\le m\}\not\subseteq S_\alpha?$$
\end{problem}

They further remarked that $S=\{n^2:n=1,2,...\}$ is not the required sequence for their problem. Since Ma and Chen did not indicate any idea or detail of their claim, the author would like to offer an alternative proof of it for convenience of the readers.
The aim is to prove that for any integer $m$, there exists $0<\alpha<1$ such that $$\{2^i:0\le i\le m\}\subseteq \left\{\left\lfloor\alpha n^2\right\rfloor:n=1,2,...\right\}.$$
The argument goes as follows: Let $\alpha$ be the number satisfying the requirement. For any $0\le i\le m$, it suffices to find some $n_i$ with $2^i=\left\lfloor\alpha n_i^2\right\rfloor$, which is equivalent to
$$\sqrt{\alpha^{-1}2^i}\le n_i<\sqrt{\alpha^{-1}(2^i+1)}.$$
The above inequalities would be satisfied if we have $$\sqrt{\alpha^{-1}}\left(\sqrt{2^i+1}-\sqrt{2^i}\right)>1 \quad (0\le i\le m).$$
The proof ends up with the choice of $\alpha=4^{-1}(2^m+1)^{-1}$.

As we mentioned in the abstract, the answer to Problem \ref{problem1} is affirmative. 
\begin{theorem}\label{thm2-1} Let 
$$S=\left\{\left\lfloor \left(\frac{3}{2}\right)^n\right\rfloor:n=1,2,3,...\right\}.$$
Then
$$\left\{2^3,2^{4}\right\}\not\subseteq S_\alpha$$
for any $0<\alpha<1$.
\end{theorem}

\begin{remark} 
For sufficiently large $n$ we have 
$$\left\lfloor \left(\frac{3}{2}\right)^n\right\rfloor<\left\lfloor \left(\frac{3}{2}\right)^{n+1}\right\rfloor\le 2\left\lfloor \left(\frac{3}{2}\right)^n\right\rfloor.$$
So, Theorem \ref{thm1} gives a positive answer to  Problem \ref{problem1} with a stronger form. Actually, our theorem can be slightly strengthen with $3/2$ replaced by a general number $\gamma\in [3/2,2)$, whereas, in the latter case, the exceptions $2^3$ or $2^4$ in the theorem should be substituted by two other suitable integers $2^j$ or $2^{j+1}$ with $j$ depending on $\gamma$.
\end{remark}

\begin{proof}[Proof of Theorem \ref{thm2-1}]
Suppose that $2^3\in S_\alpha$ for some $0<\alpha<1$, then there is a positive integer $k$ such that
$$2^3=\left\lfloor\alpha\left\lfloor \left(\frac{3}{2}\right)^k\right\rfloor\right\rfloor,$$ 
which means that
$$2^3\le \alpha\left\lfloor \left(\frac{3}{2}\right)^k\right\rfloor<2^3+1,$$
and hence 
$$2^3\alpha^{-1}\le \left\lfloor \left(\frac{3}{2}\right)^k\right\rfloor<\left(2^3+1\right)\alpha^{-1}.$$
Getting removed of the floor function, it follows that
$$2^3\alpha^{-1}\le\left(\frac{3}{2}\right)^k<\left(2^3+1\right)\alpha^{-1}+1,$$
from which we deduce that
\begin{align}\label{new1}
12\alpha^{-1}\le\left(\frac{3}{2}\right)^{k+1}<\left(12+\frac32\right)\alpha^{-1}+\frac32
\end{align}
and
\begin{align}\label{new2}
18\alpha^{-1}\le\left(\frac{3}{2}\right)^{k+2}<\left(18+\frac94\right)\alpha^{-1}+\frac94.
\end{align}
From the right--hand side inequality of (\ref{new1}), we have
$$\alpha\left\lfloor\left(\frac{3}{2}\right)^{k+1}\right\rfloor<12+\frac32+\frac{3}{2}\alpha<15.$$
From the left--hand side inequality of (\ref{new2}), we have
$$\alpha\left\lfloor\left(\frac{3}{2}\right)^{k+2}\right\rfloor>\alpha(18\alpha^{-1}-1)>17.$$
Thus,
$$\left\lfloor\alpha\left\lfloor\left(\frac{3}{2}\right)^{k+1}\right\rfloor\right\rfloor\le 14 \quad \text{whereas} \quad \left\lfloor\alpha\left\lfloor\left(\frac{3}{2}\right)^{k+2}\right\rfloor\right\rfloor\ge 17.$$
It follows that $2^4\not\in S_\alpha$ provided that $2^3\in S_\alpha$.
\end{proof}

\section*{Acknowledgments}
The author would like to thank Professor Yong--Gao Chen for his helpful comments for an early draft.  

The author is supported by National Natural Science Foundation of China (Grant No. 12201544), Natural Science Foundation of Jiangsu Province in China (Grant No. BK20210784) and China Postdoctoral Science Foundation (Grant No. 2022M710121),   the foundations of the projects "Jiangsu Provincial Double--Innovation Doctor Program'' (Grant No. JSSCBS20211023) and "Golden  Phoenix of the Green City--Yang Zhou'' to excellent PhD (Grant No. YZLYJF2020PHD051).

\end{document}